\documentclass[12pt]{article}
\usepackage{latexsym,color,amsmath,amsthm,amssymb,amscd,amsfonts}
\usepackage{tikz}
\usetikzlibrary{arrows}
\usepackage{scalefnt}
\usepackage[font=small,labelfont=bf]{caption}
\usepackage{float}
\usepackage{graphicx} 
\usepackage{amsopn}
\usepackage{relsize}
\usepackage{mathrsfs}
\usepackage{upgreek}

\renewcommand{\epsilon}{\varepsilon}
\newtheoremstyle{nonum}{}{}{\itshape}{}{\bfseries}{.}{ }{\thmnote{#3}}

\newtheorem{thm}{Theorem}[section]
\newtheorem*{thm*}{Theorem}

\newtheorem{cor}[thm]{Corollary}
\newtheorem{lem}[thm]{Lemma}

\newtheorem{rem}[thm]{Remark}

\newtheorem{exm}[thm]{Example}

\newtheorem*{definition*}{Definition}
\newtheorem{fact}[thm]{Fact}

\newtheorem*{rems*}{Remarks}

\theoremstyle{nonum}

\newcommand{\R}{\mathbb R}

\def\A{{\cal A}}
\def\L{{\cal L}}

\def\J{{\cal J}}

\def\boxL{{ \square}}
\def\boxA{{ \boxdot}}
\def\boxJ{{ \boxtimes}}

\def\calL{{\cal L}}
\def\L{{\calL}}

\def\vol{{\rm Vol}}
\def\intr{\int_{\R^n}}

\def\epi{{\rm epi}\,}
\def\epif{\epi \varphi }

\def\cvx{\rm{Cvx}}

\newcommand{\jinfp}{\diamond\kern-0.62em{\cdot}} 

\begin{document}
\title{On the Linear Structures Induced by the Four Order
Isomorphisms Acting on $\cvx_0(\R^n)$}
\date{}
\author{D.I. Florentin, A. Segal}
\maketitle
\begin{abstract}
It is known that the volume functional $\,\phi\mapsto\int e^{-\phi}\,$
satisfies certain concavity or convexity inequalities with respect to
three of the four linear structures induced by the order isomorphisms
acting on $\cvx_0(\R^n)$. In this note we define the fourth linear
structure on $\cvx_0(\R^n)$ as the pullback of the standard linear
structure under the $\J$ transform.
We show that, interpolating with respect to this linear structure, no
concavity or convexity inequalities hold, and prove that a
quasi-convexity inequality is violated only by up to a factor of $2$.
We also establish all the order relations which the four dif{}ferent
interpolations satisfy.
\end{abstract}

\section{Introduction}
In this note we study properties of the volume functional
$\vol\left(\phi\right) := \int e^{-\phi}$, defined on $\cvx_0(\R^n)$, the
class of all lower semi continuous convex functions from $\R^n$ to
$[0, \infty]$ which attain the minimal value of $0$ at the origin. As
a first example let us mention that the classical H\"older inequality
implies that $\log(\vol)$ is convex with respect to the standard linear
structure on $\cvx_0(\R^n)$, namely pointswise addition:
\begin{equation}\label{eq-holder}
\intr e^{-\left((1-\lambda) \varphi + \lambda \psi\right)} \le
\left( \intr e^{-\varphi} \right)^{1-\lambda}
\left( \intr e^{-   \psi} \right)^\lambda.
\end{equation}
We thus say that $\vol$ is log-convex, or $0$-convex, with respect to
pointwise addition. For $p\neq 0$, we say that $\vol$ is $p$-convex if
for any $\varphi,\psi\in \cvx_0(\R^n),\,\lambda\in [0,1]$ one has:
\[
\vol\left( (1-\lambda)\varphi + \lambda \psi \right) \le
\left(
(1-\lambda) \vol(\varphi)^p
  +\lambda  \vol(   \psi)^p
\right)^\frac1p.
\]
We use the notation $M_{p, \lambda}(a,b) =
\left( (1-\lambda)a^p + \lambda b^p \right)^\frac1p$ for the
$p$-average of two non-negative numbers $a,\,b$, with weights
$(1-\lambda),\,\lambda$ respectively. It is known that $\vol$ does
not satisfy any $p$-concavity inequality with respect to pointwise
addition. However, there exists additional linear structures defined
on $\cvx_0(\R^n)$, which we next describe.

In \cite{AM-Hidden}, Artstein and Milman described all the order
isomorphisms acting on the class of {\em geometric convex functions}
$\cvx_0(\R^n)$. They proved that up to linear terms, there exist only
two order reversing isomorphisms, namely the Legendre transform $\L$
and the polarity transform $\A$, and two order preserving
isomorphisms, namely the identity and the gauge transform $\J$. We
refer the reader to \cite{AM-Hidden} for the definitions and a
detailed description of these transforms. Each of the transforms
induces a linear structure on $\cvx_0(\R^n)$. For example, the
inf-convolution $\varphi \boxL \psi$ of two geometric convex
functions $\varphi,\,\psi$ may be defined as the pullback of standard
addition under the Legendre transform:
\begin{equation*}\label{eq-boxl-def}
\varphi \boxL \psi : = \L^{-1} \left( \L\varphi + \L\psi \right).
\end{equation*}
The renowned Pr\'ekopa-Leindler inequality (see \cite{Leindler,Prekopa})
states that $\vol$ is log-concave with respect to the linear structure
$\boxL$. Namely, for all $\varphi,\psi\in \cvx_0(\R^n)$ and
$\lambda \in (0,1)$:
\begin{equation}\label{eq-PL}
\intr e^{-\varphi \boxL_\lambda \psi} \ge
\left( \intr e^{-\varphi} \right)^{1-\lambda}
\left( \intr e^{-   \psi} \right)^\lambda,
\end{equation}
where $\varphi \boxL_\lambda \psi = \L^{-1}
\left( (1-\lambda)\L\varphi + \lambda\L\psi \right)$ denotes the
average of $\varphi$ and $\psi$ with respect to $\boxL$.
The linear structure induced by the polarity transform $\A$ was
defined in \cite{AFS} by
\begin{equation*}\label{eq-boxa-def}
\varphi \boxA \psi : = \A^{-1} \left( \A\varphi + \A\psi \right),
\end{equation*}
and it was proven that $\vol$ is $(-1)$-concave with respect to
$\boxA$, namely:
\begin{equation}\label{eq-PPL}
\intr e^{-\varphi \boxA_\lambda \psi} \ge
\left(
\frac{1-\lambda}{\intr e^{-\varphi}} +
\frac{  \lambda}{\intr e^{-\psi   }} 
\right)^{-1}.
\end{equation}
Here, $\varphi \boxA_\lambda \psi = \A^{-1}
\left( (1-\lambda)\A\varphi + \lambda\A\psi \right)$ denotes averaging
the functions $\varphi,\,\psi$ with respect to $\boxA$.
In an analogous way, we define here a fourth linear structure on
$\cvx_0(\R^n)$, denoted by $\boxJ$, as the pullback of standard
addition under the $\J$ transform, that is:
\begin{equation*}\label{eq-boxj-def}
\varphi \boxJ \psi : = \J^{-1} \left( \J\varphi + \J\psi \right),
\quad \mbox{ and } \quad
\varphi \boxJ_\lambda \psi : = \J^{-1}
\left(
	(1-\lambda) \J\varphi + \lambda \J\psi
\right).
\end{equation*}
In light of the log-convexity of $\vol$ with respect to the standard
linear structure (H\"older's inequality \eqref{eq-holder}), and its
log-concavity and $(-1)$-concavity with respect to $\boxL$ and
$\boxA$ respectively (inequalities \eqref{eq-PL} and \eqref{eq-PPL}),
it is natural to ask whether $\vol$ is $p$-convex, for some $p$, with
respect to $\boxJ$. In this note we provide a negative answer to this
question, but show that $\vol$ does possess some convexity property,
i.e. a quasi-convexity inequality is satisfied up to a constant (in
contrast to any type of concavity, which cannot be considered even in
this weak sense, as shown in Corollary \ref{cor-no-concavity}). More
precisely, for any $p\in [-\infty, \infty]$ we define
\[
c_p =
\sup_{\varphi,\psi,\lambda}
\left\{
\frac
{\intr e^{- \varphi \boxJ_\lambda \psi}}
{M_{p, \lambda}\left( \intr e^{-\varphi},\, \intr e^{-\psi} \right)}
\right\}.
\]
Clearly $c_p$ is decreasing in $p$, since $M_{p,\lambda}(a,b)$ is
increasing. Theorem \ref{thm-no-convex} states that $c_\infty > 1$
(which implies $c_p > 1$ for all $p$).
\begin{thm}\label{thm-no-convex}
There exist $\varphi,\psi\in \cvx_0(\R^n)$ and $\lambda\in(0,1)$ such that
\begin{equation*}\label{eq-counter-example}
\max \left\{ \intr e^{-\varphi},\, \intr e^{-\psi} \right\} <
\intr e^{- \varphi \boxJ_\lambda \psi}.
\end{equation*}
\end{thm}
However, up to a constant the answer is positive. More precisely,
$c_\infty \le 2$.
\begin{thm}\label{thm-convex-upto2}
For any $\varphi,\psi\in \cvx_0(\R^n)$ and $\lambda \in [0,1]$, we have:	\begin{equation*}
\intr e^{- \varphi \boxJ_\lambda \psi} \le 2 \cdot
\max \left\{ \intr e^{-\varphi},\, \intr e^{-\psi} \right\}.
\end{equation*}
\end{thm}

We provide the following geometric description of $\boxJ$ as the
classical operation of radial harmonic sum, applied to the
epi-graphs, namely:
\begin{fact}\label{fact-Radial-Summation}
Let $\varphi,\psi \in \cvx_0(\R^n)$, and $\lambda\in[0,1]$. Then,
\begin{equation}\label{eq-Radial-Summation}
\epi( \varphi \boxJ_\lambda \psi ) =
\left(\,\,
(1-\lambda) \epi(\varphi)^\circ + \lambda \epi(\psi)^\circ
\,\,\right)^\circ,
\end{equation}
where $\epi \phi = \left\{(x,z)\in\R^n\times\R^+ \,:\, \phi(x) < z \right\}$.
\end{fact}

The paper is organized as follows. In Section 2 we prove Fact
\ref{fact-Radial-Summation}, and prove pointwise order relations
between $\boxJ$ and the other three summations. In Section 3 we prove
Theorems \ref{thm-no-convex} and \ref{thm-convex-upto2}. In the
Appendix we compute the asymptotics of the lower bound on $c_\infty$
which was obtained in Theorem \ref{thm-no-convex}, and show that in
fact $\,c_\infty > 1 + \frac{c}{\sqrt{n}}$.

\noindent {\bf Acknowledgements:}
The authors would like to thank Artem Zvavitch for useful
discussions. The first named author was partially supported by the
U.S. National Science Foundation Grant DMS-1101636, and also
partially supported by the AMS-Simons Travel Grant, which is
administered by the American Mathematical Society with support from
the Simons Foundation.

\section{Order relations between the linear structures}
In this section we give a geometric description of $\boxJ$, and prove
four (of the possible six) order relations between the four linear
structures $+, \boxL, \boxA, \boxJ$.

In \cite{Firey}, Firey defined the polar mean $K_\lambda$ of convex
bodies $K_0, K_1 \subset \R^n,$ using Minkowski average and the
duality map $K \mapsto K^\circ :=
\left\{
	y \,:\, \forall x \in K, \langle  y, x \rangle \le 1
\right\}$:
\begin{equation*}\label{eq-dualsum}
K_{\lambda} = ((1-\lambda)K_0^\circ + \lambda K_1^\circ)^\circ.
\end{equation*}
The body $K_\lambda$ is often called the {\em radial harmonic
average} of $K_0$ and $K_1$, since:
\begin{equation*}\label{eq-rad-har}
\rho_{K_{\lambda}} = M_{-1,\lambda}
\left(
\rho_{K_0},
\rho_{K_1}
\right),
\end{equation*}
where $\rho_{K_i}$ is the radial function of $K_i$. As mentioned in
the introduction, Fact \ref{fact-Radial-Summation} states that the
operation $\boxJ_\lambda$ corresponds to radial harmonic average of
epi-graphs.

\noindent {\bf Proof of Fact \ref{fact-Radial-Summation}}.
Recall that $\epi(\phi \boxL_\lambda \eta) =
(1-\lambda)\epi(\phi) + \lambda \epi(\eta)$. In \cite{AM-Hidden} it
was shown that, denoting the reflection about $\R^n \times \{0\}$ by
$R(x,z) = (x,-z)$, we have $\epi(\A \phi) = R(\epi(\phi)^\circ)$.
Since $R(K^\circ) = (RK)^\circ$ and $\J = \A \circ \L = \L \circ \A$,
we have:
\begin{eqnarray*}
\epi( \J((1-\lambda)\J\varphi + \lambda \J \psi) ) &=&
\epi( \A \L ((1-\lambda)\L\A \varphi + \lambda \L \A \psi) ) =
\\ &=&
\epi( \A ((\A\varphi) \boxL_\lambda (\A\psi))) =
\\ &=&
R(\epi((\A\varphi) \boxL_\lambda (\A\psi))^\circ) =
\\ &=&
R ( \left(\,\,
	(1-\lambda) R(\epi(\varphi)^\circ) + \lambda R(\epi(\psi)^\circ  )
\,\,\right) ^\circ ) =
\\ &=&
\left(\,\,
(1-\lambda) \epi(\varphi)^\circ + \lambda \epi(\psi)^\circ
\,\,\right)^\circ.
\end{eqnarray*}\qed

The four linear structures $+, \boxL, \boxA, \boxJ$ of{}fer four
dif{}ferent ways to interpolate between two given functions $\varphi,
\psi \in \cvx_0(\R^n)$. We now turn to establish order relations
between the four interpolations $
\varphi \boxL_\lambda \psi,\,
\varphi \boxA_\lambda \psi,\,
\varphi \boxJ_\lambda \psi$, and
$\varphi+_\lambda \psi: = (1-\lambda)\varphi + \lambda \psi$.
First, let us consider two basic examples, one where the functions
$\varphi, \, \psi$ we average are convex indicator functions, and one
where they are norms.
\begin{exm}\label{ex-indicators}
Let $K,T$ be convex bodies such that $\varphi = 1_K^\infty,\,
\psi = 1_T^\infty \in \cvx_0(\R^n)$. Since
$\J 1_K^\infty = \|\cdot\|_K,\,
 \L 1_K^\infty = h_K,\,$ and
$\A 1_K^\infty = 1_{K^\circ}^\infty$, it may be easily checked that:
\begin{eqnarray*}
\varphi \boxA_\frac12 \psi &=& 1_{K \vee T}^\infty,
\qquad \quad
\varphi \boxJ_\frac12 \psi  =  1_{\left(\frac{K^\circ + T^\circ}{2}\right)^\circ}^\infty,
\\
\varphi \boxL_\frac12 \psi &=& 1_{\frac{K+T}{2}}^\infty,
\qquad \quad
\varphi     +_\frac12 \psi  =  1_{K \cap T}^\infty,
\end{eqnarray*}
where $K\vee T$ denotes the convex hull of $K$ and $T$. Note that:
\begin{equation}\label{eq-bodies-inc}
K \cap T
\subseteq
\left(\frac{K^\circ + T^\circ}{2}\right)^\circ
\subseteq
\frac{K+T}{2}
\subseteq
K\vee T.
\end{equation}
Thus the functions $\varphi = 1_K^\infty,\, \psi = 1_T^\infty$
satisfy the following inequalities:
\begin{equation*}
\varphi \boxA_\frac12 \psi \,\,\le\,\,
\varphi \boxL_\frac12 \psi \,\,\le\,\,
\varphi \boxJ_\frac12 \psi \,\,\le\,\,
\varphi     +_\frac12 \psi.
\end{equation*}
Note that all the inclusions in \eqref{eq-bodies-inc} are strict if
$K\neq T$. Thus for $K\neq T$ we get:
\begin{equation}\label{eq-uncompare1}
\varphi \boxL_\frac12 \psi \not\le
\varphi \boxA_\frac12 \psi,
\qquad \quad
\varphi     +_\frac12 \psi \not\le
\varphi \boxJ_\frac12 \psi.
\end{equation}
\end{exm}

\begin{exm}\label{ex-norms}
Let $K,T$ be convex bodies such that $\varphi = \|\cdot\|_K,\,
\psi = \|\cdot\|_T \in \cvx_0(\R^n)$. Since
$\J \|\cdot\|_K = 1_K^\infty,\,
 \L \|\cdot\|_K = 1_{K^\circ}^\infty,\,$ and
$\A \|\cdot\|_K = h_K$, it may be easily checked that:
\begin{eqnarray*}
\varphi \boxL_\frac12 \psi &=& \|\cdot\|_{K \vee T},
\qquad \quad
\varphi     +_\frac12 \psi  =  \|\cdot\|_{\left(\frac{K^\circ + T\circ}{2}\right)^\circ},
\\
\varphi \boxA_\frac12 \psi &=& \|\cdot\|_{\frac{K+T}{2}},
\qquad \quad
\varphi \boxJ_\frac12 \psi  =  \|\cdot\|_{K \cap T}.
\end{eqnarray*}
By \eqref{eq-bodies-inc}, the functions $\varphi = \|\cdot\|_K,\,
\psi = \|\cdot\|_T$ satisfy the following inequalities:
\begin{equation*}
\varphi \boxL_\frac12 \psi \,\,\le\,\,
\varphi \boxA_\frac12 \psi \,\,\le\,\,
\varphi     +_\frac12 \psi \,\,\le\,\,
\varphi \boxJ_\frac12 \psi.
\end{equation*}
As before, assuming $K\neq T$ we have:
\begin{equation}\label{eq-uncompare2}
\varphi \boxA_\frac12 \psi \not\le
\varphi \boxL_\frac12 \psi,
\qquad \quad
\varphi \boxJ_\frac12 \psi \not\le
\varphi     +_\frac12 \psi.
\end{equation}
\end{exm}

From \eqref{eq-uncompare1} and \eqref{eq-uncompare2} we conclude that
the functions $\varphi \boxA_\lambda \psi$ and
$\varphi \boxL_\lambda \psi$ are not comparable in general. Similarly
the functions $\varphi +_\lambda \psi,\, \varphi \boxJ_\lambda \psi$
are not comparable, in general.
However, the remaining four pairs of functions do satisfy, in
general, the order relations exhibited in the examples above.
That is, interpolating with respect to $+, \boxJ$ (induced by the two
order preserving transforms) always yields larger functions compared
to $\boxL, \boxA$ (induced by the order reversing transforms). More
precisely:
\begin{lem}
Let $\varphi, \psi \in \cvx_0(\R^n)$. Then,
\[
\max \left\{ \varphi \boxL_\lambda \psi,\, \varphi \boxA_\lambda \psi  \right\} \le
\min \left\{ \varphi     +_\lambda \psi,\, \varphi \boxJ_\lambda \psi  \right\}.
\]
\end{lem}
\begin{proof} We need to prove the following four inequalities.
\begin{itemize}
	\item[1.] $\varphi \boxL_\lambda \psi \le \varphi     +_\lambda \psi$
	\item[2.] $\varphi \boxL_\lambda \psi \le \varphi \boxJ_\lambda \psi$
	\item[3.] $\varphi \boxA_\lambda \psi \le \varphi     +_\lambda \psi$
	\item[4.] $\varphi \boxA_\lambda \psi \le \varphi \boxJ_\lambda \psi$
\end{itemize}	
The first inequality follows by letting $x=y=z$ in the definition of
the inf-convolution:
\begin{equation}\label{eq-boxL<standard}
(\varphi \boxL_\lambda \psi)(z) =
\inf_{z = (1-\lambda)x + \lambda y}
\{(1-\lambda)\varphi(x)+\lambda \psi(y)\} \le
(1-\lambda)\varphi(z)+\lambda \psi(z).
\end{equation}
For the second inequality, recall that duality is a ``convex
operation'', in the following sense (see \cite{Firey}). If $K$ and
$T$ are convex and $0 < \lambda < 1$, then:
\begin{equation}\label{eq-convex-dual}
((1-\lambda)K + \lambda T)^\circ \subseteq (1-\lambda)K^\circ + \lambda T^\circ.
\end{equation}
Combining the above with Fact \ref{fact-Radial-Summation} we get:
\begin{eqnarray*}
\epi(\varphi \boxJ_\lambda \psi) &=&
\left(\,\,
(1-\lambda) \epi(\varphi)^\circ + \lambda \epi(\psi)^\circ
\,\,\right)^\circ \,\,\, \subseteq
\\ &\subseteq&
(1-\lambda)\epi(\varphi) + \lambda \epi(\psi) \,\,\, =
\\ &=&
\epi (\varphi \boxL_\lambda \psi),
\end{eqnarray*}
which is equivalent to
\begin{equation}\label{eq-boxL<boxJ}
\varphi \boxL_\lambda \psi \leq \varphi \boxJ_\lambda \psi.
\end{equation}
The third and fourth inequalities we need to prove, are in fact
equivalent to the second and first inequalities, respectively. To see
this, first note that:
\begin{equation}\label{eq-J-boxL}
\J\left( (\J\varphi) \boxL_\lambda (\J\psi) \right) =
\varphi \boxA_\lambda \psi,
\end{equation}
\begin{equation}\label{eq-J-boxJ}
\J\left( (\J\varphi) \boxJ_\lambda (\J\psi) \right) =
\varphi +_\lambda \psi,
\end{equation}
\begin{equation}\label{eq-J-box+}
\J\left( (\J\varphi) +_\lambda (\J\psi) \right) =
\varphi \boxJ_\lambda \psi.
\end{equation}
Since $\J$ is order preserving, \eqref{eq-boxL<standard} implies
that:
\[
\varphi \boxA_\lambda \psi = 
\J\left( (\J\varphi) \boxL_\lambda (\J\psi) \right) \le
\J\left( (\J\varphi)     +_\lambda (\J\psi) \right) =
\varphi \boxJ_\lambda \psi.
\]
Similarly, \eqref{eq-boxL<boxJ} implies that:
\[
\varphi \boxA_\lambda \psi = 
\J\left( (\J\varphi) \boxL_\lambda (\J\psi) \right) \le
\J\left( (\J\varphi) \boxJ_\lambda (\J\psi) \right) = \varphi +_\lambda \psi.
\]
\end{proof}

\section{Convexity and concavity properties of $\vol$}
We begin this section by noting that the geometric Example
\ref{ex-norms} demonstrates that for any $p$, $\vol$ is not
$p$-concave with respect to $\boxJ$, i.e. $\vol$ is not even
quasi-concave. Moreover, we show that in contrast to Theorem
\ref{thm-convex-upto2}, a concavity inequality does not even hold up
to a constant. More precisely,
\begin{cor}\label{cor-no-concavity}
The functional $\vol$ is not quasi-concave with respect to $\boxJ$.
Moreover:
\[
\inf_{\varphi,\psi,\lambda}
\left\{
\frac{\vol(\varphi \boxJ_\lambda \psi)}
{\min \left\{ \vol(\varphi),\, \vol(\psi) \right\} }
\right\} = 0.
\]
\end{cor}
\begin{proof}
	We may choose $K,T$ to be boxes of unit volume, with intersection of
	arbitrarily small volume (we denote volume of sets in $\R^n$ by
	$\vol_n$). Since $\vol(\|\cdot\|_L) = n! \vol_n(L)$ for every convex
	body $L$, the functions $\varphi, \,\psi$ of Example \ref{ex-norms}
	satisfy
	\[
	\frac{\vol(\varphi \boxJ_\frac12 \psi)}
	{\min \left\{ \vol(\varphi),\, \vol(\psi) \right\} } =
	\frac{\vol_n(K \cap T)}
	{\min \left\{ \vol_n(K),\, \vol_n(T) \right\} } =
	\vol_n(K \cap T),
	\]
	which as mentioned above, can be made arbitrarily small.
\end{proof}
We now turn to prove the main results of the paper. Recall that in
\cite{AFS}, the measure $\nu$ on $\R^n \times \R^+$ was defined by
$d\nu = n(z) dxdz\,$ for $n(z) =
e^{-\frac1z} \left( \frac1z \right)^{n+2}$, and it was shown that for
any $\phi \in \cvx_0(\R^n)$
\begin{equation}\label{eq-nu-for-phi}
\nu \left( \epi \phi \right) = \intr e^{-\J \phi}.
\end{equation}
We shall rely on this fact in the proof of Theorem \ref{thm-no-convex},
when constructing a counter example to quasi-convexity of $\vol$.

\noindent {\bf Proof of Theorem \ref{thm-no-convex}.}
Let us define the (decreasing) function $G: \left[ 0, \infty \right]
\to \left[ 0, n! \right]$ by $G(t) = \int_t^\infty n(z) dz =
\int_0^\frac1t s^n e^{-s} ds$. Since $\J(\J\phi) = \phi$, we may
rewrite \eqref{eq-nu-for-phi} as
\begin{equation}\label{eq-nu-for-Jphi}
\intr e^{-\phi} =
\nu \left( \epi \J\phi \right) =
\intr dx\int_{\left( \J\phi \right) (x)}^\infty n(z)dz =
\intr G\left( \left( \J\phi \right) (x) \right) dx.
\end{equation}
Since $\J \left(\varphi \boxJ_\lambda \psi\right) =
(1-\lambda) \J\varphi + \lambda \J\psi$, the problem boils down to
convexity of $G$. Dif{}ferentiating twice we get $G''(t) = \frac
{(n+2)e^{-\frac1z}}{z^{n+4}} \left( z - \frac1{n+2} \right) $, thus
$G$ is concave on $\left[0, \frac1{n+2} \right]$. Therefore we shall
choose $\varphi,\psi$ such that $\J\varphi,\J\psi$ attain values in
$\left[0, \frac1{n+2} \right]$, to get
\begin{equation*}\label{eq-G-defies-arithmetic}
\intr e^{-\varphi \boxJ_\lambda \psi} \ge
(1-\lambda) \intr e^{-\varphi}
  +\lambda  \intr e^{-   \psi}.
\end{equation*}
If we can also choose $\varphi,\psi$ with equal volumes
$\intr e^{-\varphi} = \intr e^{-\psi}$, and such that the above
inequality is strict, the theorem will be proven. The first
condition is satisfied by choosing $\psi(x) = \varphi(-x)$. For the
second condition, it suf{}fices that $\J\varphi,\J\psi$ be supported on
the same set, attain values in $\left[0, \frac1{n+2} \right]$, and
such that $\J\varphi\neq\J\psi$ on a set of positive measure. For
example, we may choose them to be piecewise linear on the cube. More
precisely, let us denote
the unit ball of $l_\infty^n$ by $C = [-1, 1]^n$,
and the half-space $\left\{ x\in \R^n \,:\, x_1\le n+2 \right\}$ by
$H$, so that $\|x\|_H = \max\left\{0, \frac{x_1}{n+2} \right\}$.
We define $\varphi$ by setting $\J\varphi =
\max \left\{ 1_C^\infty, \|\cdot\|_H \right\}$, that is:
\[
\J\varphi (x) =
\left\{
\begin{array}{lr}
\max\left\{0, \frac{x_1}{n+2} \right\}	&: x     \in C \\
+\infty						&: x \not\in C \\
\end{array}
\right\}.
\]
By \eqref{eq-nu-for-Jphi} we have
\begin{eqnarray*}\label{eq-proving1.1}
\vol\left( \varphi \boxJ_\lambda \psi \right) &=&
\int_C G\left(\left(\J\left( \varphi \boxJ_\lambda \psi\right)\right)(x)\right) dx =
\\ \\ &=&
\int_C G\left(
			(1-\lambda)\left(\J\varphi \right)(x) + \lambda\left(\J\psi \right)(x)
		\right) dx >
\\ \\ &>&
(1-\lambda) \int_C G\left( \left(\J\varphi \right)(x) \right) dx +
   \lambda  \int_C G\left( \left(\J\psi    \right)(x) \right) dx =
\\ \\ &=&
(1-\lambda) \vol\left(\varphi\right) + \lambda \vol\left(\psi\right) =
\\ \\ &=&
\max \left\{ \vol\left(\varphi\right), \, \vol\left(\psi\right) \right\},
\end{eqnarray*}
which completes the proof. \qed

\begin{rem}\label{rem-sqrt-n}
In the Appendix we show that in fact the functions $\varphi,\psi$
satisfy:
\begin{equation}\label{eq-pretty-extremal-counter-example}
\frac{\vol\left( \varphi \boxJ_\frac12 \psi \right)}
{\max \left\{ \vol\left(\varphi\right), \, \vol\left(\psi\right) \right\}} =
1 +  \frac1{\sqrt{8\pi n}} + O\left(\frac1n\right),
\end{equation}
which implies that $c_\infty > 1 + \frac{c}{\sqrt{n}}$ for some $c>0$.
\end{rem}

\noindent {\bf Proof of Theorem \ref{thm-convex-upto2}.} Recall that
the function $G(t) = \int_t^\infty n(z)dz$ is decreasing on $[0,\infty]$,
therefore on every interval (bounded or unbounded), $G$ attains its
maximum at the lower endpoint. In particular
$G((1-\lambda)a + \lambda b) \le G(a) + G(b)$. Thus we may use
\eqref{eq-nu-for-Jphi} again, to get
\begin{eqnarray*}\label{eq-proving1.2}
\vol\left( \varphi \boxJ_\lambda \psi \right) &=&
\intr G\left(\left(\J\left( \varphi \boxJ_\lambda \psi\right)\right)(x)\right) dx =
\\ \\ &=&
\intr G\left(
(1-\lambda)\left(\J\varphi \right)(x) + \lambda\left(\J\psi \right)(x)
\right) dx \le
\\ \\ &\le&
\intr G\left( \left(\J\varphi \right)(x) \right) dx +
\intr G\left( \left(\J\psi    \right)(x) \right) dx =
\\ \\ &=&
\vol\left(\varphi\right) + \vol\left(\psi\right) \le
\\ \\ &\le&
2\cdot\max \left\{ \vol\left(\varphi\right), \, \vol\left(\psi\right) \right\}.
\end{eqnarray*}
\qed

\section*{Appendix}
We now turn to prove \eqref{eq-pretty-extremal-counter-example},
which as mentioned in Remark \ref{rem-sqrt-n}, implies that for some
$c>0$ we have $c_\infty > 1 + \frac{c}{\sqrt{n}}$. We need to compute the
volumes of the functions $\varphi,\, \psi,\, \varphi \boxJ_\lambda \psi$
defined in the proof of Theorem \ref{thm-no-convex}. Recall our
notation for the cube $C = [-1, 1]^n$ and the half-space $H =
\left\{ x\in \R^n \,:\, x_1\le n+2 \right\}$, and note that
$S_\lambda := \left( (1-\lambda)H^\circ + \lambda (-H)^\circ \right)^\circ$
is the slab $\left\{ x\in \R^n \,:\,
-\frac{n+2}{\lambda} \le x_1 \le \frac{n+2}{1-\lambda}
\right\}$.
\[
\J \left( \varphi \boxJ_\lambda \psi \right) (x) =
\left\{
\begin{array}{lr}
\max\left\{\lambda (-x_1), (1-\lambda)x_1 \right\}	&: x     \in C \\
+\infty						&: x \not\in C \\
\end{array}
\right\} =
\max \left\{ 1_C^\infty, \|\cdot\|_{S_\lambda} \right\}.
\]
Since $\J$ interchanges with taking a maximum (see \cite{AM-Hidden}),
we have:
\begin{eqnarray*}\label{eq-varphi}
\varphi =
\J \left( \max \left\{ 1_C^\infty, \|\cdot\|_H \right\} \right) =
\max \left\{ \J 1_C^\infty, \J \|\cdot\|_H \right\} =
\max \left\{ \| \cdot \|_C, \, 1_{H}^\infty \right\},
\end{eqnarray*}
and similarly $\varphi \boxJ_\lambda \psi =
\max \left\{ \| \cdot \|_C, \, 1_{S_\lambda}^\infty \right\}$.
In order to compute the volumes of $\varphi$ and
$\varphi \boxJ_\lambda \psi$ we require their level sets. At this
point we choose $\lambda = \frac12$, which maximizes
$\vol\left( \varphi \boxJ_\lambda \psi \right)$, since
$\lambda \mapsto \vol\left( \varphi \boxJ_\lambda \psi \right)$ is
concave and symmetric about $\lambda = \frac12$. We get:
\begin{eqnarray*}
L_z(\varphi) = \left( z C \right) \cap H =
\left\{
\begin{array}{lr}
z C &:  z \le n+2 \\
\left[ -z, n+2 \right] \times [-z,z]^{n-1} &:  n+2 \le z
\end{array}
\right\},
\end{eqnarray*}
\begin{eqnarray*}
L_z(\varphi \boxJ_\frac12 \psi) = \left( z C \right) \cap S_\frac12 =
\left\{
\begin{array}{lr}
z C &:  z \le 2(n+2) \\
\left[-2(n+2), 2(n+2)\right] \times [-z,z]^{n-1} &:
2(n+2) \le z
\end{array}
\right\}.
\end{eqnarray*}
With the (volumes of the) level sets known, we may use Fubini to get
\begin{eqnarray*}
\vol(\varphi) = \vol(\psi) &=& \int_{\epif} d\mu =
\int_0^\infty \vol_n\left( L_z(\varphi) \right) e^{-z} dz =
\\ \\ &=&
2^n \int_0^{n+2} z^n e^{-z} dz +
2^{n-1}\int_{n+2}^\infty (n+2 + z)z^{n-1} e^{-z} dz =
\\ \\ &=&
2^{n}\int_0^\infty z^n e^{-z} dz +
2^{n-1} \int_{n+2}^\infty (n+2 - z) z^{n-1} e^{-z} dz =
\\ \\ &=&
2^n n! - 2^{n-1} \int_{n+2}^\infty (z - (n+2)) z^{n-1} e^{-z} dz =
\\ \\ &=&
2^n n!
\left(
	1 -
	\frac{\int_{n+2}^\infty \left(\frac{z - (n+2)}{2}\right) z^{n-1} e^{-z} dz}{n!} 
\right) \equiv 2^n n! \left(1 - R(n) \right).
\end{eqnarray*}
As for $\varphi \boxJ_\frac12 \psi$, we have:
\begin{eqnarray*}
\vol(\varphi \boxJ_\frac12 \psi) &=&
\int_{\epi \left( \varphi \boxJ_\frac12 \psi \right) } d\mu =
\int_0^\infty \vol_n\left( L_z( \varphi \boxJ_\frac12 \psi ) \right) e^{-z} dz =
\\ \\ &=&
2^n \int_0^{2(n+2)} z^n e^{-z} dz +	2^n \int_{2(n+2)}^\infty 2(n+2) z^{n-1} e^{-z} dz =
\\ \\ &=&
	2^n     \int_0^{\infty} z^n e^{-z} dz -
	2^n \int_{2(n+2)}^\infty \left(z - 2(n+2)\right) z^{n-1} e^{-z} dz =
\\ \\ &=&
	2^n n!
	\left(
		1 -
		\frac{\int_{2(n+2)}^\infty \left(z - 2(n+2)\right) z^{n-1} e^{-z} dz}{n!}
	\right)	\equiv 2^n n! \left(1 - r(n) \right).
\end{eqnarray*}
Note that $\int_{n+2}^\infty z^n e^{-z} dz < n!\,$, and by
Taylor's expansion, $\frac{\left(1 + \frac2n\right)^{n+1}}{e^2} = 1 + O\left(\frac1n\right)$.
We use these estimates
together with Stirling's formula $\frac{n!}{\sqrt{2\pi n}} =
\left( \frac{n}{e} \right)^n \cdot \left(1 + O\left(\frac1n\right) \right)$
to obtain a lower bound on $R(n)$. After integration by parts we get:
\begin{eqnarray*}
R(n) &=&\frac{1}{n!}
\left(
	\frac{\left(1 + \frac2n\right)^{n+1}}{2e^2}
	\left( \frac{n}{e} \right)^n
	-\frac1n \int_{n+2}^\infty z^n e^{-z} dz
\right) =
\\ \\ &=&
\frac{\left(1 + O\left(\frac1n\right) \right)}{\sqrt{8\pi n}}
-\frac1n \frac{\int_{n+2}^\infty z^n e^{-z} dz}{n!} >
\\ \\ &>&
\frac{1}{\sqrt{8\pi n}} - \frac1{n} + O\left(\frac1{n^{\frac32}}\right) =
\frac{1}{\sqrt{8\pi n}} + O\left(\frac1n\right).
\end{eqnarray*}
To obtain an upper bound on $r(n)$, we begin again with integration by parts.
\begin{eqnarray*}
r(n) &=&\frac{1}{n!}
\left(
	\frac{\left(1 + \frac2n\right)^{n+1}}{e^3}
	\left( \frac{n}{e} \right)^n
	\left( \frac{2}{e} \right)^{n+1}
	-\left(1+\frac4n\right) \int_{2(n+2)}^\infty z^n e^{-z} dz
\right) <
\\ \\ &<&
\frac{\left(1 + O\left(\frac1n\right) \right)}{e\sqrt{2\pi n}}
\left( \frac{2}{e} \right)^{n+1}
< \left( \frac{2}{e} \right)^n.
\end{eqnarray*}
To sum it up,
\begin{eqnarray*}\label{eq-sumitup}
\frac{\vol\left( \varphi \boxJ_\frac12 \psi \right)}
{\max \left\{ \vol\left(\varphi\right), \, \vol\left(\psi\right) \right\}} &=&
\frac{1-r(n)}{1-R(n)} >
\frac
{1-\left( \frac{2}{e} \right)^n}
{1-\frac{1}{\sqrt{8\pi n}} + O\left(\frac1n\right)} =
\\ \\ &=&
\left( 1 - \left( \frac{2}{e} \right)^n\right)
\cdot
\left( 1 + \frac{1}{\sqrt{8\pi n}} + O\left(\frac1n\right) \right) =
\\ \\ &=&
\left( 1 + \frac{1}{\sqrt{8\pi n}} + O\left(\frac1n\right) \right).
\end{eqnarray*}

\bibliographystyle{amsplain}

\smallskip \noindent
Dan Florentin \\
Department of Mathematics \\
Cleveland State University, Cleveland, OH, 44115-2214, USA \\
{\it email}: danflorentin@gmail.com  \\

\smallskip \noindent
Alexander Segal \\
Afeka Academic College of Engineering, Tel Aviv, 69107, Israel\\
{\it email}: segalale@gmail.com  \\

\end{document}